\documentclass[11pt]{amsart}
\usepackage{amsmath,amsthm, amscd, amssymb, amsfonts}
\usepackage{hhline}

\usepackage[all]{xy}

\usepackage{pictexwd} %diagramas conmutativos
\usepackage{pslatex}
\usepackage{t1enc}
\usepackage{latexsym}
\usepackage[latin1]{inputenc} %esto permite (en Windows) escribir directamente � � � � � � ,etc
\usepackage{subfigure}    % From CTAN/macros/latex/contrib/supported/subfigure
\usepackage{verbatim}
\usepackage{subfigure}    % From CTAN/macros/latex/contrib/supported/subfigure
\usepackage{amsgen,amsopn}
\newtheorem{theorem}{Theorem}[section]
\newtheorem{proposition}[theorem]{Proposition}

\newtheorem{lemma}[theorem]{Lemma}
\theoremstyle{definition}
\newtheorem{definition}{Definition}[section]
\newtheorem{example}{Example}[section]

\theoremstyle{remark}
\newtheorem{remark}{Remark}[section]

\def\lg{\mathfrak{g}}

\title [Multiplier $\epsilon$-Bialgebras]{Multiplier Infinitesimal
Bialgebras and derivator Lie Bialgebras}

\author[Ochoa]{Jes\'us Alonso Ochoa Arango*}
\email{jesus.ochoa@javeriana.edu.co}
\author[Tiraboschi]{Alejandro Tiraboschi**}
\email{tirabo@famaf.unc.edu.ar}
\author[Wang]{Shuanhong Wang.***}
\email{shuanhwang@seu.edu.cn}

\address{\noindent
*Departamento de Matem\'aticas, Facultad de ciencias,
Pontificia Universidad Javeriana. Bogot\'a, Colombia.%Department of Mathematics, Southeast University, Nanjing 210096 (China).
}

\address{\noindent
** Facultad de Matem\'atica, Astronom\'\i a y F\'\i sica, Universidad
Nacional de C\'ordoba.  CIEM -- CONICET. (5000) Ciudad
Universitaria, C\'ordoba, Argentina.}

\address{\noindent ***Department of Mathematics,
Southeast University, Nanjing 210096 (China).
}

\thanks{This work was partially supported by CONICET, Fundaci\' on
Antorchas, Agencia C\'ordoba Ciencia, ANPCyT, Secyt (UNC) and Pontificia Universidad Javeriana}
\date{\today}

\begin{document}

\renewcommand{\baselinestretch}{1.2}
\thispagestyle{empty}

\begin{abstract}
We propose a definition of \emph{multiplier infinitesimal bialgebra}
and a definition of \emph{derivator Lie bialgebra}.
We give some examples of these structures and
prove that every \emph{bibalanced} multiplier infinitesimal
bialgebra gives rise to a  multiplier Lie bialgebra.
\end{abstract}

\maketitle

\section*{Introduction}
An infinitesimal bialgebra is a triple $(A, m, \Delta)$
where $(A, m)$ is an associative algebra, $(A, \Delta )$
is a coassociative coalgebra and for each $a, b \in A$,
$$\Delta(ab) = \sum a b_1 \otimes b_2 + a_1 \otimes a_2 b.$$
Infintesimal bialgebras were introduced by Joni and Rota \cite{JR}
in order to provide an algebraic framework for the calculus of
divided differences. In \cite{A2} the author introduces
several new examples, in particular he shows that
the path algebra of an arbitrary quiver admits a
canonical structure of infinitesimal bialgebra.

The notion of \emph{multiplier Hopf algebra}
was introduced by Van Daele \cite{VD1}. This concept
is a far generalization of the usual concept of Hopf algebra
to the non-unital case. In that work the author
consider an associative algebra $A$, with or without identity,
and a morphism $\Delta$ from $A$ to the multiplier
algebra $M(A \otimes A)$ of $A \otimes A$,
and imposes certain conditions on $\Delta$
(such as \emph{coassociativity}). The motivating example
was the case where $A$ is the algebra of finitely supported,
complex valued functions on an infinite discrete group,
and where $\Delta(f, g)(s, t) = f(st)$.
In these and subsequent works Van Daele together with his
collaborators have developed many of the results about
finite dimensional Hopf algebras to this context.

In this work we introduce the notion of
\emph{multiplier infinitesimal bialgebra}
(see definition \ref{e-multiplier bialgebra}) and
give several examples of the concept. This is a
mix of the two mentioned structures in which
the coproduct take values in the multiplier
algebra $M(A \otimes A)$ and at the same time
it is a derivation, in which $M(A \otimes A)$
is consider as an $A$-bimodule in the obvious way.

We also introduce the notion of \emph{derivator Lie bialgebra}
and give a couple of examples of this structure.
This is an approach to a nonassociative version
of the multiplier Hopf algebras of Van Daele.
In the last section we obtain a
result that relates the two above concepts in
a similar way to the one obtained by Aguiar in \cite{A1}.

\section*{Acknowledgements}
The first author would like to thank to Professor
Marcelo Aguiar for answer him to many questions
about infinitesimal bialgebras and
related structures. This work begin as a part of the PhD. Thesis
of the first author at Facultad de Matematica, Astronom\'ia y F\'isica
of Universidad Nacional de C\'ordoba in Argentina and
was continued while he was at Pontificia Universidad Javeriana in Colombia.
He would like thank to them for their warm and hospitality.
\section{Infinitesimal multiplier bialgebras}

From  now on we only consider $\mathbb{C}$-algebras
with non-degenerate product, that is,
$\cdot: A \times A \to A$, the product of $A$,
is no degenerate as an $A$-valued $\mathbb{C}$-bilinear form.

\noindent If $A$ is an algebra, we will denote by $L(A, \cdot)$
(resp. $R(A, \cdot)$) the space of {\em left multipliers }
(resp. {\em right multipliers }) of $A$:
\begin{align*}
L(A, \cdot) &= \{\lambda \in Hom_{\mathbb C}(A,A):
\lambda(ab) = \lambda(a)b,\;\; \forall a, b \in A \}, \\
R(A, \cdot) &= \{\rho\in Hom_{\mathbb C}(A,A):\rho(ab)
= a\rho(b),\;\; \forall a,b \in A \}.
\end{align*}
we also denote by $M(A, \cdot)$ the algebra of multipliers of $A$, i.e.
$$
M(A, \cdot) = \{(\lambda,\rho)\in Hom_{\mathbb C}(A,A)^2:
\lambda \in L(A), \rho\in R(A) \text{ and } b\lambda(c) =
\rho(b)c, \;\forall\,b,c \in A \}.
$$

\begin{remark}
Let $a \in A$, define $\lambda_a, \; \rho_a \in Hom_{\mathbb C}(A,A)$
by $\lambda_a(b) = ab$ and $\rho_a(b) = ba$.
Then $\lambda_a$ and $\rho_a$ are left and right multipliers,
respectively. Moreover, since the product of $A$ is
non-degenerated, the canonical map $a \mapsto (\lambda_a,\rho_a)$
is a one to one algebra map from $A$ to $M(A, \cdot)$.

\noindent We remind that the product in $M(A)$ is given by:
$$
(\lambda,\rho)(\lambda',\rho') = (\lambda \circ \lambda', \rho' \circ \rho).
$$
With this product, $M(A)$ is an algebra with identity.
\end{remark}

\begin{definition}\label{def-coproduct}
A linear map $\Delta: A \to M(A\otimes A) $ is called a
{\em coproduct} if
\begin{enumerate}
\item[(a)] $T_3(a\otimes b) = \Delta(b)(a \otimes 1)$ and $T_4(a\otimes b) =
(1 \otimes b) \Delta(a)$ are elements of $A \otimes A$
for all $a,b \in A$,
\item[(b)] the map $\Delta$ is coassociative in the sense that:
\begin{equation} \label{coassociativity}
(\iota \otimes T_4) \circ (T_3 \otimes \iota)
=  (T_3 \otimes \iota)\circ(\iota \otimes T_4).
\end{equation}
\end{enumerate}
\end{definition}

\begin{remark} Using the Sweedler notation, introduced
by Van Daele, the identity (\ref{coassociativity}) could be read as:
\begin{equation}\label{coassociativity2}
b_1a\otimes b_{21} \otimes c b_{22} = b_{11}a\otimes b_{12} \otimes cb_2,
\end{equation}
since
\begin{align*}
(\iota \otimes T_4) \circ (T_3 \otimes \iota) (a\otimes b \otimes c)
&=  (\iota \otimes T_4) \circ  (\Delta(b)(a\otimes 1) \otimes c) \\
&= (\iota \otimes T_4) \circ (b_1a\otimes b_2 \otimes c) \\
&= b_1a\otimes (1\otimes c)\Delta(b_2) \\
&= b_1a\otimes b_{21} \otimes c b_{22}.\\
(T_3 \otimes \iota)\circ(\iota \otimes T_4) (a\otimes b \otimes c)
&=   (T_3 \otimes \iota)\circ (a \otimes (1 \otimes c)\Delta(b)) \\
&= (T_3 \otimes \iota)\circ  (a\otimes b_1 \otimes cb_3) \\
&= \Delta(b_1)(a \otimes 1) \otimes cb_2 \\
&= b_{11}a\otimes b_{12} \otimes cb_2.
\end{align*}
The common value in (\ref{coassociativity2}) will be denoted by
$b_1a \otimes b_2 \otimes cb_3$.
\end{remark}

\begin{definition}\label{e-multiplier bialgebra}
An \emph{infinitesimal multiplier bialgebra}
(or multiplier $\epsilon$-bialgebra)
is a triple $(A,m,\Delta)$ where
\begin{enumerate}
\item $(A,m)$ is an associative algebra with non-degenerate product,
\item $\Delta: A \to M(A\otimes A) $ is a coproduct on $A$, and
\item $\Delta(ab) = \Delta(a) \circ (1 \otimes b) + (a \otimes 1) \circ \Delta(b).$
\end{enumerate}
\end{definition}

\subsection{Examples of Multipliers $\epsilon$-Bialgebras}

\begin{example}\label{quivers}{\bf Quivers. }
Let $(Q, s, t)$ be an infinite quiver.
Where $s$ and $t$ stands for the source and target maps respectively.
Let $A(Q)$ and $V(Q)$ be the set of arrows and
vertices of $Q$ respectively. Let
$$\Gamma = \{\gamma:  I \to A(Q) : \text{ where }
I \subseteq \mathbb{Z}, \text{ is an interval and }
t(\gamma(i)= s(\gamma(i+1)) \},
$$
be the set of paths in $Q$ (note that here we are considering
finite and infinite paths).

\noindent For all $\gamma \in \Gamma$ we will 
denote $I_\gamma = \gamma^{-1}(A(Q))$.
If $I_\gamma = [i,.)$ then we write $s(\gamma) = \gamma_{i} = \gamma(i)$,
the {\em source} of $\gamma$; and if
$I_\gamma = (-\infty,\cdot)$ we simply write $s(\gamma) = -\infty$.
In an analogous way, if $\gamma \in \Gamma$ and  $I_\gamma = (\cdot,i]$,
we denote $t(\gamma) = \gamma_{i} = \gamma(i)$, the {\em target} of $\gamma$.
and if $I_\gamma = (\cdot,+\infty)$ we write $t(\gamma) = +\infty$.
The length of a path $\gamma$ is defined as $|\gamma| : = |I_\gamma|$,
the cardinal of the set $I_\gamma$.

For any vertex $e \in V(Q)$ we will define two sets
\begin{center}
 $\Gamma_e = \{ i \in Dom \gamma: \;s(\gamma_i)= e \}$
and $\Gamma^e= \{ i \in Dom \gamma: \; t(\gamma_i) = e\}$.
\end{center}

\noindent Let $\Gamma' =\{\gamma \in \Gamma:\;
\forall e \in V(Q),\;|\Gamma_e|< \infty \;\text{and}\;
|\Gamma^e|< \infty\} \cup \{ \pm \infty \}$
for all $e \in V(Q)\} \cup\{-\infty,+\infty\}$.
Thus  $\Gamma'$ is the set of  paths $\gamma \in \Gamma$,
such that for every $e \in V(Q)$, $\gamma$ pass trough $e$
only finitely times and we adjoint to this set
two elements $\{-\infty,+\infty\}$ that doesn't belong to $V(Q)$
and we define $s(\pm\infty) = t(\pm\infty) = \pm\infty$.

\begin{definition}
The {\em generalized path algebra}  of $Q$ is the
vector space $k_\infty Q$ with basis $\Gamma'$
and where the multiplication is the concatenation
of paths whenever is possible and zero otherwise.
The added symbols $\pm \infty$ are idempotents by definition.
\end{definition}
\begin{remark}
We can observe here that when we define the set $\Gamma'$
we adjoint two elements $\pm \infty$, the reason for this
is that we want to have a no degenerated product in the generalized
path algebra $k_\infty Q$.
\end{remark}

\vskip .5cm

We are going to define a coproduct on $k_{\infty}Q$
in the following way.
If $\gamma = \cdots \gamma_{i}\gamma_{i+1}\cdots$, then
$$
\Delta(\gamma) = \sum_{i \in \mathbb Z} \cdots \gamma_{i-2}\gamma_{i-1}
\otimes \gamma_{i+1}\gamma_{i+2}\cdots;
$$
that is, for  $\gamma \in \Gamma$ we define
$\Delta(\gamma) = (\Delta_{\lambda}(\gamma), \Delta_{\rho}(\gamma))
\in M(A \otimes A)$ by the formulas:
$$
\Delta_{\lambda}(\gamma)(\gamma' \otimes \gamma'') =
\sum_{i \in \mathbb Z} \cdots \gamma_{i-2}\gamma_{i-1} \gamma'
\otimes \gamma_{i+1}\gamma_{i+2} \cdots \gamma''.   (*)
$$
$$
\Delta_{\rho}(\gamma)(\gamma' \otimes \gamma'') =
\sum_{i \in \mathbb Z} \gamma' \cdots \gamma_{i-2}\gamma_{i-1}
\otimes \gamma''\gamma_{i+1}\gamma_{i+2} \cdots.   (**)
$$
Since $\Gamma^e$ and $\Gamma_e$ %$\{ i: e_{i-1} = s(\gamma')\}$ and $\{ i: t(\gamma'')  =  e_{i}\}$
are finite (by definition of $\Gamma'$) then both sums are well defined.
We define $\Delta(\pm \infty) = 0$ and extend $\Delta$
linearly to the whole $A$.
It is clear that $\Delta_{\lambda}$ and $\Delta_{\rho}$
are left and right multipliers,
respectively, and that
$(\gamma_1 \otimes \gamma_2) \Delta_{\lambda}(\gamma)(\gamma_3 \otimes \gamma_4)
=  \Delta_{\rho}(\gamma)(\gamma_1 \otimes \gamma_2)(\gamma_3 \otimes \gamma_4)$;
thus  $\Delta(\gamma)  \in M(A \otimes A)$.

That $\Delta$ is a derivation, i.e.,
\begin{equation}\label{delta derivation quiver}
\Delta(\gamma \cdot \gamma') = (\gamma \otimes 1)\Delta(\gamma')
+ \Delta(\gamma)(1 \otimes \gamma'),
\end{equation}
follows by direct verification  from the possible form
of the paths. For instance, if $\gamma = \cdots \gamma_{n-1}\gamma_n$
then, we have two cases.
\begin{enumerate}
 \item If $e_n$ the target of $\gamma_n$ is
different of $e'_0$, the source of $\gamma'_0$,
both sides of \eqref{delta derivation quiver} are zero.
\item If $e_n = e'_0$, then% $\gamma_n = \gamma'_0$, then
$$
\Delta(\gamma \cdot \gamma') =  \sum_{i \le n}
(\cdots \gamma_{i-1} \otimes \gamma_{i+1} \cdots
\gamma_n \gamma'_1 \gamma'_2 \cdots)
+  \sum_{1\le i} (\cdots \gamma_n \gamma'_1 \gamma'_2
\cdots \gamma'_{i-2}\gamma'_{i-1}
\otimes \gamma'_{i+1}\gamma'_{i+2}\cdots).
$$
On the other hand
\begin{align*}
(\gamma \otimes 1)\Delta(\gamma')
&=  (\gamma \otimes 1)\sum_{1 \le i} (\gamma'_0
\cdots \gamma'_{i-1} \otimes \gamma'_{i+1}\gamma'_{i+2} \cdots)  \\
&= \sum_{1 \le i}( \cdots \gamma_{n-1}\gamma_n \gamma'_1
\cdots \gamma'_{i-1} \otimes \gamma'_{i+1}\gamma'_{i+2} \cdots),
\end{align*}
and
\begin{align*}
\Delta(\gamma) (1 \otimes \gamma')
&= \sum_{ i \le n}( \cdots \gamma_{i-1} \otimes
\gamma_{i+1} \cdots \gamma_n)  (1 \otimes \gamma') \\
&=\sum_{ i \le n}( \cdots \gamma_{i-1}
\otimes \gamma_{i+1} \cdots \gamma_n \gamma'_1 \cdots),
\end{align*}
so the desired conclusion follows.
\end{enumerate}

It is not hard to check  that $\Delta(\gamma)(\gamma' \otimes 1)
\in A \otimes A$ and
$(1 \otimes \gamma )\Delta(\gamma') \in A \otimes A$.
\end{example}

\begin{example}\label{Posets}{\bf Posets. }

Let $L$ be a poset and let  $\mathcal P$ be the set of
subposets of $L$ with maximum and minimum.
If $P \in \mathcal P$ we denote by $1_P$ and $0_P$
the maximum and the minimum of $P$ respectively.

Denote by $A_{\mathcal P}$ the $k$-vector space
with base $\mathcal P$ equipped with the following product:
$$
P*Q = \left\{ \begin{matrix} P \cup Q &\quad \text{ if }
1_P = 0_Q \quad \text{with the inherited order,}\\ 0 & \quad \text{ otherwise.}\end{matrix} \right.
$$
for every $P, Q \in \mathcal P$ and extend
this product linearly to the whole $A_{\mathcal P}$.

\begin{remark}
The algebra $A_{\mathcal P}$ is nonunital but have local units.
In fact, if $P_1,P_2,\ldots,P_n \in \mathcal P$
then write $S = \{0_{P_1}, \cdots, 0_{P_n}\}$
and $T = \{1_{P_1}, \cdots, 1_{P_n}\}$. It is clear
that $\sum_{u \in S} \{u\}$ and $\sum_{v \in T} \{v\}$
are left and right local units, respectively,
for any linear combination of
$\{ P_1,P_2,\ldots,P_n \}$.
Thus we have that the product of $A_{\mathcal P}$
is no degenerated and the canonical map from
$A_{\mathcal P}$(or $A_{\mathcal P} \otimes A_{\mathcal P}$)
to its multipliers is injective. We also
have that $(A_{\mathcal P})^2 = A_{\mathcal P}$.
\end{remark}

\noindent \emph{Notation:}
if $P \in \mathcal P$ we will denote $P_0 = P-\{1_P\}$.
For every $P \in  \mathcal P$ and $x \in P_0$
we will to write $(-\infty, x]_{P} := \{y \in P:  y \le x\}$
and $[x,+\infty)_{P} := \{y \in P:   x \le y\}$.
If $x \not\in P_0$  then $(- \infty, x]_{P} :=0$
and $[x, + \infty)_{P} := 0$.

\vskip 0.5cm

Let $\Delta:  A_{\mathcal P} \to M( A_{\mathcal P}\otimes A_{\mathcal P})$
be the map given by
$$\Delta_{\lambda}(P)(Q\otimes R)
= \sum_{x \in P_0} (-\infty, x]*Q \otimes [x, +\infty)*R
$$%= (-\infty, 0_Q]*Q \otimes [0_Q, +\infty)*R $$
and
$$\Delta_{\rho}(P)(Q\otimes R)
= \sum_{x \in P_0} Q*(-\infty, x] \otimes R*[x, +\infty)
%%= Q*(-\infty, 1_R] \otimes R*[1_R, +\infty).
$$

It is clear that $\Delta_{\lambda}(P)$ and
$\Delta_{\rho}(P)$ are left and right multipliers
of $ A_{\mathcal P}\otimes A_{\mathcal P}$ respectively,
and it easy to check that $(P \otimes Q) \Delta_{\lambda}(T)(R \otimes S)
= \Delta_{\rho}(T)(P \otimes Q)(R \otimes S)$.

\begin{lemma} Let $P,S \in \mathcal P$. Then,
$$
\Delta(P)(S \otimes 1) \in  A_{\mathcal P}\otimes A_{\mathcal P}
\quad \text{ and } \quad (1 \otimes S)\Delta(P) \in
A_{\mathcal P}\otimes A_{\mathcal P}.
$$
\end{lemma}
\begin{proof}
We are going to show that if $P, S \in \mathcal P$ then
$$
(\Delta_{\lambda}(P) \circ (S \otimes 1)\,,\,(S \otimes 1)
\circ \Delta_{\rho}(P)) = (\lambda_U,\rho_U),
$$
where $U = (-\infty, 0_S]_P * S \otimes [0_S , +\infty )_P$.
In fact, if $0_{S} \in P_0$ then
\begin{align*}
\Delta_{\lambda}(P) \circ (S \otimes 1)(X \otimes Y)
&= \sum_{x \in P_0} (-\infty, x]*S*X \otimes [x, +\infty)*Y\\
&= (-\infty, 0_{S}]*S*X \otimes [0_{S}, +\infty)*Y,\quad (*)\\
(S \otimes 1) \circ \Delta_{\rho}(P)(X \otimes Y)
&= \sum_{x \in P_0} X*(-\infty, x]*S \otimes Y*[x, +\infty)\\
&= X*(-\infty, 0_S]*S \otimes Y*[0_S, +\infty), \quad(**)
\end{align*}
and
\begin{align*}
\lambda_U(X \otimes Y) &= (- \infty, 0_{S}]*S*X
\otimes [0_{S}, +\infty)*Y, \\
\rho_U(X \otimes Y) &= X*(-\infty, 0_S] * S \otimes Y*[0_S, +\infty).
\end{align*}
The left and right multiplier of $U$ are identical to $(*)$
and $(**)$ respectively.
\noindent In the same way we show that $(1 \otimes S)\Delta(P)
\in  A_{\mathcal P}\otimes A_{\mathcal P}$.
\end{proof}

Now we will to probe that $\Delta$ is a derivation.
If $P*Q = 0$, then $1_P \not= 0_Q$ and then follows
that $0=\Delta(P)(1\otimes Q) = (P \otimes 1)\Delta(Q)$.
If $P*Q \not= 0$ then $1_P = 0_Q$ and
$(P *Q)_0 = P_0 \cup Q_0$, thus
\begin{align*}
\Delta(P*Q) &=  \sum_{x \in (P*Q)_0} (-\infty, x]_{P*Q}
\otimes [x, +\infty)_{P*Q}  \\
&=  \sum_{x \in P_0} (-\infty, x]_{P*Q} \otimes
[x, +\infty)_{P*Q} +\sum_{x \in Q_0} (-\infty, x]_{P*Q}
\otimes [x, \infty)_{P*Q} \\
&= \sum_{x \in P_0} (-\infty, x]_{P}
\otimes [ x, +\infty)_{P*Q} +\sum_{x \in Q_0} (-\infty, x]_{P*Q}
\otimes [x, +\infty)_{Q}  \\
&= \sum_{x \in P_0} (-\infty, x]_{P} \otimes
[x, +\infty)_{P}*Q +\sum_{x \in Q_0} P*(-\infty, x]_{Q}
\otimes [x, +\infty)_{Q}  \\
& =  \Delta(P)(1\otimes Q) + (P \otimes 1)\Delta(Q).
\end{align*}
\end{example}
%%%%%%%%%%%%%%%%%%%%%%%%%%%%%%%%%%%%%%%%%%%%%%%%%%%%%%%%5

\section{Derivator Lie bialgebras}

\begin{definition}\label{Derivator Lie Bialgebras}
A \emph{derivator Lie bialgebra} is a collection
$(\mathfrak{g}, [\cdot,\cdot], \delta, \zeta, T_1, T_2)$ where:
\begin{itemize}
 \item $(\mathfrak{g}, [\cdot,\cdot])$ is a Lie algebra,
 \item $\delta, \zeta: \mathfrak{g} \to Der(\mathfrak{g},
       \mathfrak{g} \otimes \mathfrak{g})$ are derivations and
 \item $T_1, T_2: \mathfrak{g} \otimes \mathfrak{g} \to \mathfrak{g} \otimes \mathfrak{g} $
       are linear maps.
\end{itemize}
These data subject to the following axioms:
\begin{itemize}
 \item {\bf Antisymmetry}. The linear maps  $\delta, \zeta$
are \emph{antisymmetric}, that is, for all $x \in \mathfrak{g}$
\begin{center} $\tau  \circ  \delta_a = -
\delta_a$ and $\tau  \circ  \zeta^b = - \zeta^b$.
\end{center}
\item {\bf Generalized CoJacobi}. For all $a,b, x \in \mathfrak{g}$,
\begin{equation}\label{CoJacobi}
(Id + \sigma + \sigma^{2})(\zeta^b \otimes Id)(T_1(a \otimes x))
=  (Id + \sigma + \sigma^{2})(\delta_a \otimes Id)(\tau(T_2(x \otimes b))).
\end{equation}
\end{itemize}
Where $\zeta^a$ and $\delta_b$ stand for $\zeta(a)$
and $\delta(b)$ respectively,
$\sigma = (Id \otimes \tau)(\tau \otimes Id)$
is the cyclic permutation and
$\tau :\lg \otimes \lg \to \lg \otimes \lg$ is the flip  map.

The pair $(\delta, \zeta)$ will be called the \emph{co-bracket} of
$\mathfrak{g}$ and the maps $T_1, T_2$
\emph{intertwining operators}.
\end{definition}

\subsection{First examples of derivator Lie bialgebras}

We remind here the definition of Lie bialgebra.
\begin{definition}
A \emph{Lie bialgebra} over $k$, is a collection
$(\lg, [ \cdot, \cdot ], \delta)$ where
$(\lg, [ \cdot, \cdot ])$ is a Lie $k$-algebra,
$\delta: \lg \to \lg \otimes \lg$ is a derivation
and it satisfies the following conditions,
\begin{itemize}
 \item ({\bf Antisymmetry}) $\tau \circ \delta = - \delta$ and
\item ({\bf CoJacobi})$(Id + \sigma + \sigma^2)(\delta \otimes Id)\delta = 0$.
\end{itemize}
\end{definition}
\begin{example}
If $\lg$ is a one-dimensional Lie algebra
all the Lie bialgebra structures are trivial,
that is, all are equal to zero.
\end{example}

\begin{example}
Let $\lg$ be a simple Lie algebra and $\Delta$
a root system for $\lg$. We define
$r = \sum_{\alpha \in \Delta^{+}} e_\alpha \wedge e_{-\alpha}$.
The internal derivation determined by $r$
is a Lie bialgebra structure over $\mathfrak{g}$.
\end{example}

\begin{example}
Let $(\lg, [\cdot, \cdot], \beta)$ be a Lie bialgebra
and consider $\alpha: \lg \to k$ a linear functional.
We define
\begin{center}
$\delta, \zeta: \mathfrak{g} \to
Der(\mathfrak{g}, \mathfrak{g} \otimes \mathfrak{g})$ by $\delta(x)
= \zeta(x) = \alpha(x)\beta$
\end{center}
and
\begin{center}
$T_1,T_2 :\lg \otimes \lg \to \lg \otimes \lg$ by
$T_1(a \otimes x) = \alpha(a)\beta(x)$ and
$T_2(x \otimes b) = \alpha(b)\beta(x)$.
\end{center}
The collection $(\lg, [\cdot, \cdot], \delta, \zeta, T_1, T_2)$
is a derivator Lie bialgebra.
\end{example}

\begin{example}
Let $\mathfrak{g}$ be a Lie algebra of dimension $2$
and let $\{ X, Y \}$ be a base of $\lg$ such that $[X, Y] = X$.
We consider the derivation $\beta: \lg \to \lg \otimes \lg$
defined by $\beta(X) = X \wedge Y $ and $\beta(Y) = 0$.
Define a pair of operators
$\iota_X: \lg \to \lg \wedge \lg$ and
$\iota_Y: \lg \to \lg \wedge \lg$
by $\iota_X(Z) = Z \wedge X$ and $\iota_Y(Z) = Z \wedge Y$,
respectively.
We define $\delta = \zeta$ as $\delta(X) = \iota_Y$
and $\delta(Y) = X$, then we extend by linearity.
Clearly $\delta$ is a $Der(\lg, \lg \otimes \lg)$ valued
derivation of $\lg$. If we take
$T_1 =T_2 = Id: \lg \otimes \lg \to \lg \otimes \lg$,
then $(\lg, [\cdot, \cdot], \delta, \delta, Id, Id)$
is a derivator Lie bialgebra;
the verification is left to the reader.
The only that we need to check is that
each sides of equality \eqref{CoJacobi} are equals to zero.
\end{example}

\section{From infinitesimal multiplier bialgebras
to derivator Lie bialgebras}

In this section we describe a way to obtain
derivator Lie bialgebras in a similar way in which
in \cite{A1} were obtained Lie bialgebras from
infinitesimal bialgebras.

\begin{definition}\label{bibalanceato}
Let $(A, \mu, \Delta)$ be an infinitesimal multiplier
bialgebra. The \emph{bibalanceator} of $A$ is the linear map
$\mathbb{B}: A \to End(A \otimes A)^2$, where
$\mathbb{B}(a) = (\underline{\mathbb{B}}(a) , \overline{\mathbb{B}}(a))$
with
\begin{equation}\label{underbalanceator}
 \underline{\mathbb{B}}(a)(x \otimes y) = (x  \otimes 1)
\tau\Delta(y)(1\otimes a) - \tau \Delta(y)(1 \otimes ax) + (1 \otimes
y)\Delta(x)(a \otimes 1) - \Delta(x)(ay \otimes 1),
\end{equation}
and
\begin{equation}\label{overbalanceator}
 \overline{\mathbb{B}}(a)(x \otimes y) = (xa \otimes 1)
\tau \Delta(y) - (a \otimes 1)\tau\Delta(y)(1 \otimes x) + (1
\otimes ya) \Delta(y) - (1 \otimes a) \Delta(x)(y \otimes 1) .
\end{equation}
We say that $\mathbb{B}$ is \emph{symmetric} if
$\underline{\mathbb{B}}(a) = \underline{\mathbb{B}}(a)
\circ \tau$ and $\overline{\mathbb{B}}(a) = \overline{\mathbb{B}}(a)
\circ \tau$, for all $a \in A$.
We say that $A$ is \emph{bibalanced} if its bibalanceator is equal
to zero.
\end{definition}

\begin{proposition}\label{prop-bibalanceador-symetrico}
Let $(A,\mu, \Delta)$ be an infinitesimal multiplier
bialgebra. We define the following linear maps
\begin{equation}\label{inf-undercobracket}
\delta: A \to Hom_k(A, A \otimes A ), \quad x \mapsto
\delta_{a}(x):= \Delta(x)(a \otimes 1) -\tau(\Delta(x)(a \otimes
1)),
\end{equation}
\begin{equation}\label{inf-overcobracket}
 \zeta: A \to Hom_k(A, A \otimes A ), \quad
x \mapsto \zeta^{a}(x):= (1 \otimes a)\Delta(x) - \tau((1 \otimes
a)\Delta(x)).
\end{equation}
Then,
\begin{equation}\label{der-inf-undercobracket}
\delta_a[x,y] = x \cdot \delta_a(y) - y \cdot \delta_a(x) +
\underline{\mathbb{B}}(a)(y \otimes x) - \underline{\mathbb{B}}(a)(x
\otimes y),
\end{equation}
and
\begin{equation}\label{der-inf-overcobracket} \zeta^a[x,y] = x
\cdot \zeta^a(y) - y \cdot \zeta^a(x) + \overline{\mathbb{B}}(a)(y
\otimes x) - \overline{\mathbb{B}}(a)(x \otimes y).
\end{equation}
\end{proposition}
\begin{proof}
If $a, x, y \in A$ then
\begin{align*}
 \delta_a[x,y] &= \Delta([x,y])(a \otimes 1)
- \tau \Delta([x,y](a \otimes 1)) \\
& = \Delta(xy - yx)(a \otimes 1) - \tau (\Delta(xy - yx)(a \otimes
1)),
\quad \text{and since $\Delta$ is a derivation,}\\
&= (x \otimes )\Delta(y)(a \otimes 1) + \Delta(x)(a \otimes y) - (y
\otimes 1)\Delta(x)(a \otimes 1)
- \Delta(y)(a\otimes x)\\
&-\tau((x \otimes )\Delta(y)(a \otimes 1)) -\tau(\Delta(x)(a \otimes
y)) +\tau((y \otimes 1)\Delta(x)(a \otimes 1)) +\tau(\Delta(y)(a
\otimes x)).
\end{align*}
On the other side,
\begin{align*}
 x \cdot \delta_a(y) &= (ad_x \otimes 1 + 1 \otimes ad_x) \delta_a (y) \\
&=(x \otimes 1)\delta_a(y) - \delta_a(y)(x \otimes 1) + (1 \otimes
x)\delta_a(y)
- \delta_a(y)(1 \otimes x) \\
&= (x \otimes 1)\Delta(y)(a \otimes 1) - \tau((1 \otimes x)
\Delta(y)(a \otimes 1))
-\Delta(y)(ax \otimes 1) + \tau(\Delta(y)(a \otimes x))\\
&+(1 \otimes x)\Delta(y)(a \otimes 1) - \tau((x \otimes
1)\Delta(y)(a \otimes 1)) -\Delta(y)(a \otimes x) +
\tau(\Delta(y)(ax \otimes 1)).
\end{align*}
We also have,
\begin{align*}
 y \cdot \delta_a(x) &=(y \otimes 1)\Delta(x)(a\otimes 1)
-\tau((1 \otimes y)\Delta(x)(a \otimes 1))
-\Delta(x)(ay \otimes 1) + \tau(\Delta(x)(a \otimes y)) \\
&+ (1 \otimes y )\Delta(x)(a \otimes 1) - \tau((y \otimes
1)\Delta(x)(a \otimes 1)) - \Delta(x)(a \otimes y) +
\tau(\Delta(x)(ay \otimes 1)),
\end{align*}
then, by definition of bi-balanceator, we obtain
(\ref{der-inf-undercobracket}). In the same way we
obtain equation (\ref{der-inf-overcobracket}).
\end{proof}
\begin{theorem}\label{principal final}
Let $(A,\mu, \Delta)$ be an infinitesimal multiplier bialgebra.
If the bi-balanaceator of $A$ is symmetric then the linear maps
\begin{equation}\label{inf-undercobracket-1}
\delta: A \to Der(A, A \otimes A ), \quad x \mapsto \delta_{a}(x):=
 (1 \otimes a)\Delta(x) - \tau((1 \otimes
a)\Delta(x)),
\end{equation}
\begin{equation}\label{inf-overcobracket-1}
 \zeta: A \to Der(A, A \otimes A ), \quad
x \mapsto \zeta^{a}(x):= \Delta(x)(a \otimes 1) -
\tau(\Delta(x)(a \otimes 1)),
\end{equation}
are well defined, where $\delta_a := \delta(a)$ and $\zeta^{a} :=
\zeta(a)$. Moreover, if we define
\begin{center}
$T_1, T_2: A^{Lie} \otimes A^{Lie} \to A^{Lie} \otimes A^{Lie} $
by $T_1(u \otimes v) = (1 \otimes u) \Delta(v)$ and
$T_2(u \otimes v) = \Delta(u)(v \otimes 1)$,
\end{center}
then the collection $(A,\; m - m \circ \tau,\; \delta,\;
\zeta, T_1, T_2)$ is a derivator Lie bialgebra.
\end{theorem}
\begin{proof}
Since the bi-balanceator of $A$ is symmetric, by proposition
\ref{prop-bibalanceador-symetrico}, the maps
$\delta$ and $\zeta$ are derivations that
take values in the space of derivations
of $A^{Lie}$ with values in the $A^{Lie}$-module $A
\otimes A$. Then, the only that we have to prove
ins that the generalized coJacobi condition is satisfied.
In fact if $a, b, x \in A$ the using Sweedler notation
\begin{align}
(\zeta^b \otimes Id)(T_1(a \otimes x))&=
(\Delta \otimes Id)((1 \otimes a)\Delta(x))(b \otimes 1 \otimes 1) \\
&-(\tau \otimes id)((\Delta \otimes Id)((1 \otimes a)\Delta(x))(b \otimes 1 \otimes 1) \nonumber\\
&= x_{(1)(1)}b \otimes x_{(1)(2)} \otimes a x_{(2)} -
x_{(1)(2)} \otimes x_{(1)(1)}b \otimes a x_{(2)}. \nonumber
\end{align}
Now,
\begin{align}
(\delta_a \otimes Id)(\tau (T_2 (x \otimes b)))&=
( 1 \otimes a \otimes 1)(\Delta \otimes Id)( \tau(\Delta(x)(b \otimes 1)))\\
&-(\tau \otimes id)((1 \otimes a \otimes 1)(\Delta \otimes Id)(\tau(\Delta(x)(b \otimes 1))))\nonumber\\
&= x_{(2)(1)} \otimes a x_{(2)(2)} \otimes x_{(1)} b -
 a x_{(2)(2)} \otimes x_{(2)(1)} \otimes x_{(1)} b; \nonumber
\end{align}
then, since $\Delta$ is coassociative \eqref{coassociativity},
then $x_{(1)(1)}b \otimes x_{(1)(2)} \otimes ax_{(2)} =
x_{(1)}b \otimes x_{(1)(2)} \otimes ax_{(2)}$.
It follows that
$$(Id + \sigma + \sigma^{2})(\zeta^b \otimes Id)(T_1(a \otimes x))
=  (Id + \sigma + \sigma^{2})(\delta_a \otimes Id)(\tau(T_2(x \otimes b)))$$.
\end{proof}
\begin{remark}
The definition of the coproduct in equations \eqref{der-inf-undercobracket}
and  \eqref{der-inf-overcobracket} can be something
annoying since the presence of the factors $(a \otimes 1)$
and $(1 \otimes a)$, but this is because the
coassociativity constraints \ref{def-coproduct}.% and \ref{e-multiplier bialgebra}.
\end{remark}

\begin{definition}\label{generalized e-multiplier bialgebra}
A \emph{generalized infinitesimal multiplier bialgebra}
(or generalized multiplier $\epsilon$-bialgebra)
is a quadruple $(A, \ast, \cdot,\Delta)$ where
\begin{itemize}
\item $(A,\ast)$ and $(A, \cdot)$ are associative algebras;
\item $(A, \cdot)$ has non-degenerate product;
\item $\mathbb{M}(A \otimes A, \cdot)$ is a $(A, \ast)$-bimodule;
\item $\Delta: A \to M(A\otimes A, \cdot) $ is a coproduct on $A$, and
\item $\Delta(a \ast b) = a \vartriangleright \Delta(b) + \Delta(a) \vartriangleleft b.$
Where $\vartriangleright$ and $\vartriangleleft$ stand for
the left and right actions of $(A, \ast)$ on the multipliers
of $(A \otimes A,\cdot \otimes \cdot)$.
\end{itemize}
\end{definition}
\begin{remark}
In definition \ref{e-multiplier bialgebra} when $\ast$
coincide with $\cdot$ we only write $(A,\cdot,\Delta)$
for the resultant structure.
\end{remark}
\begin{remark}
 It is clear that if $(A, \Delta)$ is an $\epsilon$-bialgebra
in the sense of Aguiar \cite{A2}, with nondegerenated product then,
it is a multiplier infinitesimal bialgebra
in the sense of definition \eqref{e-multiplier bialgebra}.
In fact, the only that we need to specify
is the actions of $A$ on the space of multipliers $\mathbb{M}(A \otimes A)$
and they are given by
$$a \vartriangleright \mu = (a \otimes 1) \mu
\quad \text{and} \quad
\mu \vartriangleleft a = \mu(1 \otimes a),$$
for every $a \in A$ and $\mu \in \mathbb{M}(A \otimes A)$.
\end{remark}

\begin{example} { \bf Infinite Cyclic Group }

Let us consider $X = \{ a \}$ and let $F = < X >$
be the infinite cyclic group generated $X$ (i. e. $F$ is the
infinite cyclic group). Let $k$ be a field and let
$A = k F$  be the group algebra of $F$ over $k$.
Define $\Delta: k F \to kF \otimes kF$ by the
following formulas,
\begin{align*}\label{freegroupepsilonbialgebra}
&\Delta(e) = 0, \quad \Delta(a) = e \otimes e,\quad
\Delta(a^{-1}) = - a^{-1} \otimes a^{-1}, \\
&\Delta(a^n) = (a \otimes 1) \Delta(a^{n-1}) + e \otimes a^{n-1}, \; \text{for any }\; n >1,  \\
&\Delta(a^{-n}) = - (a^{-n} \otimes 1)\Delta(a^n) (1 \otimes a^{-n}), \; \text{for any }\; n >1.
\end{align*}

\begin{lemma}\label{lemma infinite cyclic group}
The collection $(kF, m, \Delta)$ is a $\epsilon$-bialgebra.
\end{lemma}
\begin{proof}
First, it is clear that $(k F, \Delta)$ is an
associative coalgebra. In fact, an induction argument
shows that the formulas given for $\Delta$ reduces to,
\begin{equation}\label{freegroupepsilonbialgebra1}
\Delta(a^{n+1}) = e \otimes a^n + a \otimes a^{n-1}+
\cdots a^{n-1} \otimes a + a^n \otimes e, \quad \text{for all}\; n >0,
\end{equation}
\begin{equation}
\Delta(a^{-n}) = - (a^{-n} \otimes a^{-1} + a^{-(n-1)} \otimes a^{-2} +
\cdots + a^{-2} \otimes a^{-(n-1)} + a^{-1} \otimes a^{-n}),
\quad \text{for all}\; n>0.
\end{equation}
To prove the compatibility condition, between
$m$ and $\Delta$ we use induction again.
\begin{itemize}
\item Given $m > n >0$ we have that,
$$
\Delta (a^m a^{-n}) = \Delta(a^{m-n}) =
(a \otimes e)\Delta(a^{m-n-1}) + e \otimes a^{m - n - 1};
$$
now, by induction hypothesis
$\Delta(a^{m-n-1}) = (a^{m-1} \otimes e)\Delta(a^{-n})
+ \Delta(a^m)(e \otimes a^{-n})$, then
$$\Delta(a^{m-n}) = (a^m \otimes e)\Delta(a^{-n})
+ (a \otimes e)\Delta(a^{m-1})(e \otimes a^{-n})
+ e \otimes a^{m-n-1}.$$
On the other hand,
\begin{align*}
&(a^m \otimes e)\Delta(a^{-n}) + \Delta(a^m)( e \otimes a^{-n}) \\
&= (a^m \otimes e)( -(a^{-n} \otimes e)\Delta(a^n))
 + \Delta(a^m)(e \otimes a^{-n}) \\
 &= -(a^{m-n} \otimes e)\Delta(a^n)(e \otimes a^{-n})
 + (a \otimes e)\Delta(a^{m-1})(e \otimes a^{-n})
 - (a^{m-1} \otimes e)\Delta(a^n)(e \otimes a^{-n}),
\end{align*}
and the conclusion follows.

\item For $0 < n < m $ a similar argument shows that
$\Delta(a^m a^{-n}) = (a^m \otimes e)\Delta(a^{-n})
 + \Delta(a^m)(e \otimes a^{-n}).$
\end{itemize}

In conclusion, $(kF, m, \Delta)$ is an $\epsilon$-bialgebra.
\end{proof}

%Now we proceed to construct a multiplier $\epsilon$-bialgebra
%based on the above construction.
Let us denote by $K(F)$ the space of finitely supported complex
valued functions defined on $F$. For any $n \in \mathbb{Z}$
we will to write $\delta_n$ for the function defined
over $F$ that takes the value $1$ at $a^n$ and is
equals to zero en other case.
It is clear that every element $f \in K(F)$
can be writen as the \emph{finite} linear combination
$\sum_{n \in \mathbb{Z}} f(a^n)\delta_n$.
Over $K(F)$ we can define two structures of
associative algebra, the first one is the
classical structure defined by  the formula
$\delta_m \cdot \delta_n =  \delta_{n,m} \delta_m$,
and the second one is defined in the following
\begin{lemma}\label{nonstandar product over K(F)}
Let $\ast: K(F) \otimes K(F) \to K(F)$ be the
bilinear map defined as follows
\begin{itemize}
 \item $\delta_m \ast \delta_n = \delta_{m+n+1}$, for
every $m,n \in \mathbb{Z}_{\geq 0}$,
\item $\delta_m \ast \delta_n = -\delta_{m+n+1}$, for every
$m,n \in \mathbb{Z}_{<0}$, and
\item $\delta_m \ast \delta_n = 0$ otherwise.
\end{itemize}
With these map $K(F)$ becomes an associative nonunital
algebra.
\end{lemma}
\begin{proof}
 Straightforward.
\end{proof}
\begin{remark}
It is \emph{clear} that the space of multipliers of
$(K(F) \otimes K(F), \cdot \otimes \cdot)$ is isomorphic to $C(F \times F)$,
the space of all complex valued functions defined
over $F$. Then $\mathbb{M}(K(F) \otimes K(F), \cdot \otimes \cdot)$
is a $(K(F),\ast)$-module.
\end{remark}
\begin{proposition}
Let $F$ be the free group in one generator and
let us denote by $K(F)$ the space of finitely supported complex
valued functions defined on $F$. With the same
notation of the above paragraph we define
\begin{equation}
 \Delta:K(F) \to C(F \times F) \quad
 \text{by} \quad \Delta(f)(a^m, a^n) = f(a^{m+n}),
\end{equation}
for every $f \in K(F)$ and $m,n \in \mathbb{Z}$.
Then $(K(F), \ast, \cdot, \Delta)$ is a generalized multiplier
$\epsilon$-bialgebra.
\end{proposition}
\begin{proof}
 The proof relies on the fact that $(K(F), \ast, \cdot, \Delta)$
is esentially dual to $(kF, m, \Delta)$ of lemma
\ref{lemma infinite cyclic group} and we don't include it here.
\end{proof}

\end{example}

\end{document}